\documentclass[11pt,a4paper]{article}
\usepackage{amsmath,amssymb,amsfonts,amsthm}

\newcommand{\E}{{\bf{E}}}
\newcommand{\PP}{{\bf{P}}}
\newcommand{\Var}{{\bf{Var}}}
\newcommand{\PPP}{{\bar {\bf{P}}}}
\newcommand{\EEE}{{\bar {\bf{E}}}}

\newtheorem{tm}{Theorem}

\newtheorem{lem}{Lemma}
\theoremstyle{definition}

\begin{document}
\bibliographystyle{plain}
\parindent=0pt
\centerline{\LARGE \bfseries  A note on loglog distances in a power law}
\centerline{\LARGE \bfseries random intersection graph}

\par\vskip 3.5em

\centerline{ Mindaugas Bloznelis}

\bigskip

\ \ \ Faculty of Mathematics and Informatics, Vilnius University,
LT-03225 Vilnius, Lithuania

\par\vskip 3.5em

\begin{abstract}
We consider the typical distance between vertices of the giant
component of a random intersection graph having a power law
(asymptotic) vertex degree distribution with infinite second
moment. Given two vertices from the giant component we construct
$O_P(\log \log n)$ upper bound for the length of the shortest path
connecting them.
\par\end{abstract}

\smallskip
{\bfseries key words:}  intersection graph, random graph, power law

AMS 2010 Subject Classification: Primary 05C80, Secondary 05C82
\par\vskip 5.5em
\section{Introduction}
Given a collection of subsets  $S(1),\dots, S(n)$ of the set $W=\{w_1,\dots, w_m\}$ define
the intersection graph on the vertex set $V=\{v_1,\dots, v_n\}$ such that  $v_i$ and $v_j$ are joined by an edge (denoted $v_i\sim v_j$)
 whenever $S(i)\cap S(j)\not=\emptyset$, for $i\not=j$.
Assuming that the sets $S(i)$, $i=1,\dots, n$, are drawn at random we obtain a
random intersection graph.

Random intersection graphs have applications in various fields: design and analysis of secure wireless sensor networks
\cite{{eschenauer2002}}, \cite{dipietro2004},
modelling of social networks \cite{Deijfen}, statistical clasification \cite{godehardt2003}, see also \cite{JKS}, \cite{karonski1999}.
Usually, in applications the number of interacting nodes (vertices) is large and it is convenient to
study the statistical properties of parameters of interest.

We consider  a class of random intersection graphs, where $m$ is much larger than $n$ and where the
random subsets $S(i)$, $i=1,\dots, n$, are independent.
Moreover, we assume that for every $i$, the distribution
of $S(i)$ is a mixture of uniform distributions.
That is, for every $k$, conditionally on the event $|S(i)|=k$ the
random set $S(i)$ is uniformly distributed in the class of all
subsets of $W$ of size $k$. In particular, with  $P_{*i}$ denoting
the distribution of $|S(i)|$ we have, for every $A\subset W$,
$\PP(S(i)=A)={{m}\choose{|A|}}^{-1}P_{*i}(|A|)$. The random
intersection graph corresponding to the sequence of distributions $\PP_*=(P_{*1},\dots, P_{*n})$ is denoted $G(n,m,\PP_*)$.

Assuming that as $n,m\to \infty$  the  asymptotic
distributions of $\sqrt{n/m}\, |S(i)|$ have power tails and infinite second moment we obtain the  random intersection
graph  $G(n,m,\PP_*)$ with  asymptotically heavy tailed vertex degree
distribution without second moment,
see \cite{Deijfen} and \cite{Bloznelis2007/2}.

It is known that in some random graph models  with a heavy tailed vertex degree distribution the typical distance between vertices of the giant component is of order
$O_P(\log \log n)$, see \cite{Chung-Lu2003}, \cite{Hofstad2004}, \cite{Reittu-Norros2002},
\cite{Reittu-Norros2005}, \cite{Reittu-Norros2004}.
In the present note we extend this  bound to the random intersection graph model with heavy
tailed vertex degree
distribution without second moment.

The paper is organized as follows: results are stated in Section 2. Proofs are given in Section 3.

\section{Results}\label{procedura}

Given an integer sequence $\{m_1,m_2,\dots\}$, let  $ \{(Z_{n1},\dots, Z_{nn}), n=1,2,\dots \}$ be a sequence of random vectors
with independent coordinates such that for every $n$, $Z_{ni}$ takes values in $\{0,1,\dots, m_n\}$, $1\le i\le n$.
Let $P_{ni}$ denote the distribution of $Z_{ni}$.
Write
$\PP_n=(P_{n1},\dots, P_{nn})$.
Fix two countable sets $\{v_1,v_2,\dots\}$ and $\{w_1,w_2,\dots\}$ and define the sequence  of
random intersection graphs $\{G_n=G(n,m_n,\PP_n), n=1,2,\dots\}$ as follows. Given $n$, let
$S_n(v_1),\dots, S_n(v_n)$ be independent subsets of
  $W_n=\{w_1,\dots, w_{m_n}\}$ of sizes $Z_{ni}=Z_n(v_i):=|S_{n}(v_i)|$, $1\le i\le n$,  such that
  $\PP(S_n(v_i)=A)={{m_n}\choose{|A|}}^{-1}P_{ni}(|A|)$, for  $A\subset W_n$.
$G_n$ is the graph on the
vertex set $V_n=\{v_1,\dots, v_n\}$, where $v_i$ and $v_j$ are adjacent whenever $S_n(v_i)\cap S_n(v_j)\not=\emptyset$.
   Let ${\tilde P}_{ni}$ denote  the distribution of the random variable
   ${\tilde Z}_{ni}={\tilde Z}_n(v_i):=|S_n(v_i)|\sqrt{n/m_n}$.

 Let $d_n(u,v)$ denote the distance between vertices $u,v\in V_n$ in $G_n$ ($=$number of edges
in the shortest path of $G_n$ connecting $u$ and $v$).
Let $C_1=C_1(G_n)\subset V_n$ denote the vertex set of the largest connected component of $G_n$. Therefore,
the subgraph of $G_n$ induced by $C_1$ is connected and the number of vertices of any other connected subgraph of $G_n$
is not greater than
$|C_1|$. A vertex  $u\in V_n$ is called  maximal in $G_n$ if $Z_n(u)=\max_{v\in V_n}Z_n(v)$.

\begin{tm}\label{T1}
Let $0<\alpha<1$  and $c_0,c_1,c_2>0$. Let $\{\omega_1,\omega_2,\dots\}$ be
a sequence of positive numbers satisfying $\lim_n\omega_n=+\infty$.
 Let
$\{G(n,m_n,\PP_n), n=1,2,\dots\}$ be a sequence of random intersection graphs such that

(i) $n\ln^2n=o(m_n)$ as $n\to \infty$;

(ii) $\exists$ $n_0$ such that $\forall$ $n>n_0$  we have
\begin{equation}\label{Y1}
c_1t^{-1-\alpha}\le \PP({\tilde Z}_{ni}>t)\le c_2t^{-1-\alpha},
\qquad
  \forall t \in [c_0,n^{1/(1+\alpha)}\omega_n],
  \quad
   \forall i\in \{1,\dots, n\}.
\end{equation}
Let $\{u_n\}$ be a sequence of maximal vertices, i.e., for every
$n$, the vertex $u_n\in V_n$ is maximal in $G_n$. For every
$\varepsilon>0$ we have  as $n\to \infty$
\begin{eqnarray}\label{Egle}
&&
\PP\Bigl(d(v_1,u_n)\le (1+\varepsilon)\ln^{-1} (1/\alpha)\, \ln(\ln(2+n))\, \Bigr| \,  d(v_1,u_n)<\infty\Bigr)\to 1,
\\
\label{Egle44}
&&
\PP\Bigl(d(v_1,v_2)\le (2+\varepsilon)\ln^{-1} (1/\alpha)\ln(\ln(2+n))\, \Bigr| \,  v_1,v_2\in C_1 \Bigr)\to 1.
\end{eqnarray}

Here $'\ln\,'$ denotes the natural logarithm.\end{tm}

It follows from (\ref{Egle44}), by the symmetry, that given  two vertices $v,v'$ drawn uniformly at random from
the giant component $C_1$ we have $d(v,v')=O_P(\ln\ln n)$.
Recall  that such a distance is
of much larger order $O_P(\ln n)$ in
the corresponding Erd\H os-R\'enyi graph ($G(n,p)$ with $1<c_1\le
np\le c_2$).
 This remarkable difference is explained by an effect of very
large nodes whose degrees realize the extremes
from a power law distribution, see \cite{Reittu-Norros2002}, \cite{Reittu-Norros2005}.

 Note that with probability tending to $1$
(with high probability) every maximal vertex belongs to the giant component $C_1$. In addition, as $n\to \infty$ we have
$|C_1| >\rho n$,
for some $\rho\in (0,1)$.
We collect these statements in Remark 1.

{\it Remark 1.}
{\it
Assume that conditions of Theorem \ref{T1} are satisfied. Then
\begin{equation}\label{ro-11}
\exists \, \rho\in (0,1)
\qquad
{\text{such that}}
\qquad
\PP(|C_1|>\rho\, n)\to 1
\qquad
{\text{as}}
\qquad
n\to \infty.
\end{equation}
Let $\{u_n\}$ be a sequence of maximal vertices, i.e., for every
$n$, the (random) vertex $u_n\in V_n$ is maximal in $G_n$.  Then
\begin{equation}\label{Remark-1}
\PP\bigl(u_n\in C_1(G_n)\bigr)\to 1
\qquad
{\text{as}}
\qquad
n\to \infty.
\end{equation}
}

{\it Acknowledgement.} I  would like to thank Ilkka Norros for valuable discussion.

\section{Proofs}

We start with auxiliary Lemmas \ref{L1}-\ref{L3}. Then we prove Remark 1, see Lemma \ref{L5} below, and  Theorem 1.

In what follows we write $l_2(n):=\ln(\ln(n))$, where $\ln$ denotes the natural logarithm. $H_{j,k,m}$ denotes the hypergeometric random variable with parameters $j,k\le m$  and the distribution
$\PP(H_{j,k,m}=r)=\frac{{{k}\choose{r}}{{m-k}\choose{j-r}}}{{{m}\choose{j}}}$.

\begin{lem}\label{L1}
Let $S_1,S_2$ be independent random subsets of the set $W=\{1,\dots, m\}$ such that
$S_1$ (respectively $S_2$) is uniformly distributed in the class of subsets of $W$ of size $j$ (respectively $k$).
Then $H=|S_1\cap S_2|$ is the hypergeometric random variable with parameters $j,k,m$ and mean $\E H=jk/m$.
The probability $p':=\PP(H=0)=(m-k)_j/(m)_j$ satisfies, for $j+k<m$,
\begin{equation}\label{Z1Z2}
1-\frac{jk/m}{1-(j+k)/m}\le p'\le 1-\frac{jk}{m}+\bigl( \frac{jk}{m} \bigr)^2.
\end{equation}
Here we denote $(m)_j=m(m-1)\cdots(m-j+1)$. For $0<s<1$ and $j+k\le s\, m$ we have
\begin{equation}
\label{Z1Z2A}
\frac{jk}{m}+\frac{2}{1-s}\bigl(\frac{jk}{m}\bigr)^2
\ge
\PP(S_1\cap S_2\not=\emptyset)
\ge
\frac{jk}{m}-\bigl(\frac{jk}{m}\bigr)^2.
\end{equation}
For $\lambda=\E H$ and $t\ge 0$ we have
\begin{equation}\label{Lema1-3}
\PP(H\ge\lambda+t)\le \exp\bigl\{-\frac{t^2}{2(\lambda+t/3)}\bigr\},
\qquad
\PP(H\le\lambda-t)\le \exp\bigl\{-\frac{t^2}{2\lambda}\bigr\}.
\end{equation}
In particular, we have
\begin{equation}\label{Lema1-4}
\PP(H=0)\le e^{-jk/2m}.
\end{equation}
\end{lem}

\begin{proof}[Proof of Lemma \ref{L1}]
Inequalities (\ref{Z1Z2}) are shown in \cite{JKS}. Inequalities (\ref{Z1Z2A}) are simple consequences of (\ref{Z1Z2}).
We only show the left-hand side inequality for $i,j\ge 1$. In this case  $j+k\le 2jk$ and we have
\begin{displaymath}
a:=\frac{1}{1-(j+k)/m}=1+\frac{j+k}{m}a\le 1+\frac{2jk}{m}\frac{1}{1-s}.
\end{displaymath}
Now, desired inequality follows from the left-hand side inequality (\ref{Z1Z2}).

Exponential inequalities for hypergeometric probabilities (\ref{Lema1-3}) can be derived from
the corresponding inequalities for binomial
probabilities, see \cite{Hoeffding1963}. Their proof can be found in, e.g., \cite{janson2001}.
The right-hand side inequality (\ref{Lema1-3}) applied to $t=\E H$ gives (\ref{Lema1-4}).
\end{proof}

\begin{lem}\label{L2}
Given integer $m$ and constants $0<\gamma_1< \gamma_2<1$ let
$z_1,z_2,\dots, z_r$ be integers such that $z=\sum_{h=1}^rz_h\le\
\gamma_1 m$ and $z_h\ge 6\gamma_2(\gamma_2-\gamma_1)^{-2}\ln n\ge
1$, for $1\le h\le r$. Let $S_1,S_2,\dots, S_r$ be independent
random subsets of $W=\{1,\dots, m\}$ such that, for every $h$,
$S_h$ is uniformly distributed in the class of subsets of $W$ of
size $z_h$. Then
\begin{equation}\label{D-Z}
\PP\Bigl(\bigl|\cup_{i=1}^rS_i\bigr|\ge (1-\gamma_2)\sum_{i=1}^r|S_i|\Bigr)\ge 1-rn^{-3}.
\end{equation}
\end{lem}

\begin{proof}[Proof of Lemma \ref{L2}]
Write $D_{[0]}=\emptyset$ and, for $h\ge 1$, denote $D_{[h]}=\cup_{k\le h}S_k$ and $S'_h=S_h\setminus D_{[h-1]}$.
Note that  $|D_{[r]}|=\sum_{h=1}^{r}|S'_h| \le z$. In order to prove  (\ref{D-Z}) we show that
uniformly in $h$ and $D_{[h-1]}$ (satisfying $|D_{[h-1]}|\le \gamma_1 m$)
we have
$p_h:=\PP(|S'_h|\le (1-\gamma_2) z_h\,\bigr|\, D_{[h-1]})\le n^{-3}$.
It is convenient to write this probability in the form $p_h=\PP(H \ge \gamma_2 a)$,
where $H$ denotes the hypergeometric random variable with
parameters $a=z_h$, $b=|D_{[h-1]}|$ and $m$.
We have $\E H=ab/m\le \gamma_1 a$. An application of (\ref{Lema1-3})
shows $p_h\le \exp\{-  a (\gamma_2-\gamma_1)^2/(2\gamma_2)\}$.
For $a=z_h\ge (6\gamma_2/(\gamma_2-\gamma_1)^2)\ln n$ we obtain $p_h\le n^{-3}$, thus completing the proof.
\end{proof}

\begin{lem}\label{L2+}
Given integers $1\le a,b,d\le m$,  let ${\cal S}_a\subset {\cal S}_d$ be subsets of
the set $W=\{1,2,\dots, m\}$ of sizes $|{\cal S}_a|=a$ and
$|{\cal S}_d|=d$. Here $a\le d$. Let $S_b$ be a random subset of $W$ uniformly
distributed over the subsets of $W$ of size $b$. For integers $0<s\le
r<t$ satisfying $s\le a\wedge b$, we have
\begin{equation}\label{lydeka1}
\PP\Bigl(|S_b\cap {\cal S}_d|\ge t \, \Bigr|\, |S_b\cap {\cal S}_a|\ge s\Bigr)
\le
\max_{s\le i\le r}\PP(H_{b-i,d-a,m-a}\ge t-i)+\frac{\PP(H_{a,b,m}> r)}{\PP(H_{a,b,m}\ge s)}.
\end{equation}
Assume that  $d\le m/100$. Then we have
\begin{equation}\label{lydeka2}
\PP\Bigl(|S_b\cap {\cal S}_d|\ge b/2 \, \Bigr|\, S_b\cap {\cal S}_a\not=\emptyset \Bigr)
\le
e^{-b/8}\bigl(1+4\frac{m}{ab}{\mathbb I}\{a>b/4, ab\le m, b\ge 3 \}\bigr).
\end{equation}
\end{lem}

\begin{proof}[Proof of Lemma \ref{L2+}] Let us prove (\ref{lydeka1}).
Introduce events ${\mathbb B}=\{|S_b\cap {\cal S}_d|\ge t\}$, ${\mathbb A}=\{|S_b\cap {\cal S}_a|\ge s\}$ and write
$p:=\PP({\mathbb B}|{\mathbb A})$. Denote $p_i=\PP(H_{b-i,d-a,m-a}\ge t-i)$.
Let $\sum_j$ denote the sum over subsets ${\cal A}_j\subset{\cal  S}_a$ of size $|{\cal A}_j|=j$. We have
\begin{eqnarray}\nonumber
\PP({\mathbb B}\cap{\mathbb A})
&&
=\sum_{s\le j\le a\wedge b}\sum_j\PP\bigr({\mathbb B}\cap \{S_b\cap {\cal S}_a={\cal A}_j\}\bigr)
\\
\nonumber
&&
=
\sum_{s\le j\le a\wedge b}\sum_j\PP\bigr({\mathbb B}\bigr|S_b\cap {\cal S}_a={\cal A}_j\bigr)\PP(S_b\cap {\cal S}_a={\cal A}_j)
\\
\nonumber
&&
=
\sum_{s\le j\le a\wedge b}p_j\sum_j\PP(S_b\cap {\cal S}_a={\cal A}_j)=\sum_{s\le j\le a\wedge b}p_j\PP(H_{a,b,m}=j)
\\
\label{lydeka}
&&
\le
\max_{s\le i\le r}p_i\, \PP(s\le H_{a,b,m}\le r)+\PP(H_{a,b,m}>r).
\end{eqnarray}
(\ref{lydeka1}) follows from (\ref{lydeka}) and the identity $p=\PP({\mathbb B}\cap{\mathbb A})/\PP({\mathbb A})$.

Let us prove (\ref{lydeka2}). Put $t= \lceil b/2 \rceil$, $s=1$ and $r= \lfloor b/4 \rfloor$ and apply (\ref{lydeka1}). We obtain
\begin{equation}\label{lydeka5}
\PP\Bigl(|S_b\cap {\cal S}_d|\ge b/2 \, \Bigr|\, S_b\cap {\cal S}_a\not=\emptyset \Bigr)
\le
\max_{1\le i\le r}p_i+p^*_1/p^*_2.
\end{equation}
Here we denote $p^*_1:=\PP(H_{a,b,m}> r)$, $p^*_2=\PP(H_{a,b,m}\ge 1)$.
Let us show that
\begin{equation}\label{lydeka6}
p_i\le e^{-b/8},
\qquad
1\le i\le r.
\end{equation}
For this purpose we apply the first inequality of (\ref{Lema1-3}). Denote $\lambda_i= \E H_{b-i,d-a,m-a}$ and $t_i=t-i-\lambda_i$.
We have, for $1\le i\le r$ and $d/m\le 100$,
\begin{eqnarray}\nonumber
&&
\lambda_i=(b-i)(d-a)/(m-a)\le bd/m\le b/100,
\\
\nonumber
&&
t_i\ge\lceil b/2 \rceil-\lfloor b/4 \rfloor-(b/100)\ge (b/4)-(b/100),
\\
\nonumber
&&
t_i\le \lceil b/2 \rceil-i\le b/2.
\end{eqnarray}
These inequalities combined with the inequality, which follows from  (\ref{Lema1-3}), $p_i\le e^{-t_i^2/(2(\lambda_i+t_i/3))}$
imply (\ref{lydeka6}).
Note that, for $a\le b/4$, we have  $p^*_1=0$ and, therefore,
(\ref{lydeka2}) follows from (\ref{lydeka6}) and (\ref{lydeka5}).

Now assume that $a>b/4$.  Denote $\lambda_*=\E H_{a,b,m}$ and $t_*=r+1-\lambda_*$. We have
\begin{displaymath}
\lambda_*=ab/m\le b/100,
\qquad
1+b/4> t_*> b/4-b/100.
\end{displaymath}
Note that $b\ge 3$ implies $t_*<(7/12)b$.
These inequalities combined with the inequality, which follows from  (\ref{Lema1-3}), $p^*_1\le e^{-t_*^2/(2(\lambda_*+t_*/3))}$
imply
\begin{equation}\label{lydeka8}
p^*_1\le e^{-b/8}.
\end{equation}
Finally, we apply (\ref{Lema1-4}) to get the lower bound
\begin{equation}\label{lydeka9}
p^*_2\ge 1-e^{-ab/2m}\ge ab/(4m),
\end{equation}
for $ab<m$. Invoking  (\ref{lydeka6}, \ref{lydeka8}, \ref{lydeka9}) in
(\ref{lydeka5}) we obtain (\ref{lydeka2}).

\end{proof}

\begin{lem}\label{L3}
Let $0<\alpha<1$ and $c_0,c_1,c_2>0$. Let $\{\omega_n\}$ be a positive sequence  satisfying $\omega_n\to+\infty$ as $n\to \infty$.
 Let $\{({\tilde Z}_{n1},\dots, {\tilde Z}_{nn})\}$ be a sequence
of random vectors with independent non-negative coordinates satisfying condition (ii) of Theorem 1.
We have as $n\to \infty$
\begin{equation}\label{Lema3-1}
\PP\bigl(n^{1/(1+\alpha)}/\omega_n <\max_{1\le i\le n}{\tilde Z}_{ni}\le n^{1/(1+\alpha)}\, \omega_n\bigr) \to 1.
\end{equation}

Let $L_n(t)=\sum_{i=1}^n{\tilde Z}_{ni}{\mathbb I}_{\{t<{\tilde Z}_{ni}\le n^{1/(1+\alpha)}\,\omega_n\}}$.
There exists an integer $n_1\ge n_0$ depending on $\alpha, c_1,c_2$ and the sequence $\{\omega_n\}$ such that,
for $n>n_1$ and $t\in (c_0,n^{1/(1+\alpha)})$, we have
\begin{equation}\label{Lema3-2a}
c_1/2\le \frac{\alpha}{1+\alpha}\frac{t^{\alpha}}{n}\, \E L_n(t)\le c_2,
\end{equation}
For $1<\tau<1+\alpha$ there exists an integer $n_2\ge n_0$ and number $c^*>0$ both
depending on $\alpha, \tau, c_1,c_2$ and the sequence $\{\omega_n\}$ such that,
for $n>n_2$ and $t\in (c_0,n^{1/(1+\alpha)})$, we have
\begin{equation}\label{Lema3-2b}
\PP\bigl( \bigl|L_n(t)-\E L_n(t)\bigr| > \gamma \E L_n(t)\bigr)
\le
c^*\gamma^{-\tau}n^{1-\tau}t^{(\tau-1)(\alpha+1)}.
\end{equation}
\end{lem}

\begin{proof}[Proof of Lemma \ref{L3}] The proof is routine. We include it for the sake of completeness.
Denote for short
$t_*=n^{1/(1+\alpha)}/\omega_n$ and $T_*=n^{1/(1+\alpha)}\,\omega_n$.

Let us prove (\ref{Lema3-1}).  Write
$p_{n}^*(t):=\PP(\max_{1\le i\le n}{\tilde Z}_{ni}\le t)$.
It follows from (\ref{Y1})  as $n\to \infty$
\begin{eqnarray}\label{e-1}
&&
p_n^*(t_*)
=\prod_i\PP({\tilde Z}_{ni}\le t_*)\le \bigl(1-c_1/t_*^{1+\alpha}\bigr)^n
\le \exp\{-c_1n/t_*^{1+\alpha}\}=o(1),
\\
\label{e-2}
&&
p_n^*( T_*)
=
\prod_i\PP({\tilde Z}_{ni}\le T_*)
\ge \bigl(1-c_2/T_*^{1+\alpha}\bigr)^n
\ge \exp\{-c_2n/(T_*^{1+\alpha}-c_2)\}=1-o(1).
\end{eqnarray}
In (\ref{e-1}) we apply $1-x\le e^{-x}$ to $x=c_1/t_*^{1+\alpha}$.
In (\ref{e-2}) we apply $1-y\ge e^{-y/(1-y)}$ to $y=c_2/T_*^{1+\alpha}<1$.

Let us prove (\ref{Lema3-2a}-\ref{Lema3-2b}). Given $1\le \tau< 1+\alpha$ and $n$, write
$a^{(\tau)}_i(t)=\E {\tilde Z}_{ni}^{\tau}{\mathbb I}_{\{t<{\tilde Z}_{ni}\le T_*\}}$.
It follows from  (\ref{Y1}) and the  identity
\begin{displaymath}
a^{(\tau)}_i(t)=t^{\tau}\PP(t<{\tilde Z}_{ni}\le T_*)+\tau\int_{t}^{T_*}x^{\tau-1}\PP(x<{\tilde Z}_{ni}\le T_*)dx
\end{displaymath}
that, for sufficiently large $n$ and  $t\in (c_0,n^{1/(1+\alpha)})$,
\begin{equation}\label{a1/a2(2)}
c_1/2 \
\le
\ a^{(\tau)}_i(t)\ t^{1+\alpha-\tau} \ \frac{1+\alpha-\tau}{1+\alpha} \
\le
\ c_2.
\end{equation}
Note that the right hand side inequality holds for $n>n_0$, while the left hand side inequality holds
for $n>n_0'$, where $n_0'=n_0'(\alpha,\tau,c_1,c_2,\{\omega_n\})\ge n_0$.

An application of (\ref{a1/a2(2)}) to $\E L_n(t)=\sum_{i=1}^na^{(1)}_i(t)$ shows (\ref{Lema3-2a}).

Let us show (\ref{Lema3-2b}).
Denote
$T_i={\tilde Z}_{ni}{\mathbb I}_{\{t<{\tilde Z}_{ni}\le T^*\}}-\E {\tilde Z}_{ni}{\mathbb I}_{\{t<{\tilde Z}_{ni}\le T^*\}}$.
Write, for short, $b:=\gamma\E L_n(t)$.
By Chebyshev's  inequality,
\begin{equation}
p_n(t):=\PP\bigl(  \bigl|\sum_{i=1}^nT_i\bigr|\ge b\bigr)
\le
b^{-\tau}\E \bigl|\sum_{i=1}^n T_i\bigr|^{\tau}.
\end{equation}
Invoking the inequalities
$\E|\sum T_i|^{\tau}\le \sum\E|T_i|^{\tau}$ and $\E|T_i|^{\tau}\le 2a^{(\tau)}_i(t)$, $1\le \tau \le 2$,
we obtain
\begin{equation}\label{a1/a2}
p_n(t)\le 2b^{-\tau}\sum_{i=1}^na^{(\tau)}_i(t).
\end{equation}
It follows from (\ref{a1/a2(2)}) that $\sum_{i=1}^na^{(\tau)}_i(t)\le c_2\frac{1+\alpha}{1+\alpha-\tau}\frac{n}{t^{1+\alpha-\tau}}$.
Substitution of this inequality and of (\ref{Lema3-2a}) in (\ref{a1/a2}) gives
\begin{displaymath}
p_n(t)\le  \frac{8}{1+\alpha-\tau}\frac{c_2}{c_1^\tau} \frac{1}{\gamma^{\tau}} \frac{t^{(\tau-1)(\alpha+1)}}{n^{\tau-1}}
\end{displaymath}
thus proving (\ref{Lema3-2b}).
\end{proof}

\begin{lem}\label{L5}
Assume that conditions of Theorem 1 are satisfied. Then (\ref{ro-11}) holds.
\linebreak Let $V_n^0=\{v_i:{\tilde Z}_{ni}>n^{1/(1+\alpha)}/l_2^{\alpha}(n)\}\subset V_n$. We have
as $n\to \infty$
\begin{eqnarray}
\label{L5-C1}
&&
\PP\bigl(V_n^0\subset C_1(G_n)\bigr)\to 1,
\\
\label{L5-V0}
&&
\PP\bigl(|V_n^0|\ge 2c_2(l_2(n))^{\alpha(1+\alpha)}\bigr)\to 0.
\end{eqnarray}
Here $|V_n^0|$ denotes the number of elements of the set $V_n^0$ and $l_2(n)$ denotes $\ln(\ln(n)$.
\end{lem}
Observe that (\ref{Lema3-1}) implies that every maximal vertex of $G_n$ belongs whp to $V_n^0$. Therefore, (\ref{L5-C1}) combined with
(\ref{Lema3-1}) imply (\ref{Remark-1}).

\begin{proof}[Proof of Lemma \ref{L5}]
Let us prove (\ref{L5-V0}). Write $t_{0n}=n^{1/(1+\alpha)}/l_2^{\alpha}(n)$.
We have
\begin{equation}\label{t0-v0}
|V_n^0|=\sum_{i=1}^n{\mathbb I}^0_i,
\qquad
{\mathbb I}^0_i:={\mathbb I}_{\{ {\tilde Z}_{ni}>t_{0n}\}},
\qquad
1\le i\le n.
\end{equation}
For $i=1,\dots, n$, let ${\mathbb I}_i^+$  be independent Bernoulli random variables
with success probability  $\PP({\mathbb I}_i^+=1)=c_2t_{0n}^{-1-\alpha}$.
It follows from (\ref{Y1}), (\ref{t0-v0}) that the random variable $L^+:=\sum_{1\le i\le n}{\mathbb I}_i^+$
is stochastically larger than $|V_n^0|$. Therefore, for every $a>0$ we have
\begin{displaymath}
\PP(|V_n^0|\ge a)\le \PP(L^+\ge a).
\end{displaymath}
Recall that exponential inequalities (\ref{Lema1-3}) remain valid if we replace the hypergeometric random variable $H$ by a
Binomial random variable, see  e.g., \cite{janson2001}. The first inequality of (\ref{Lema1-3}) applied to
Binomial probability $\PP(L_+\ge a)$ with  $a=2\E L^+=2c_2(l_2(n))^{\alpha(1+\alpha)}$ shows (\ref{L5-V0}).

Let us prove (\ref{ro-11}). Let $G_n^0$ be the subgraph of  $G_n$ obtained by deleting the edges incident to vertices from
$V_n^0$. Note that $G_n^0$ is a random intersection graph defined by the random sets $S^0_n(v_i)$, $v_i\in V_n$, such that
$S_n^0(v_i)=S_n(v_i)$ for ${\tilde Z}_{ni}\le t_{0n}$ and $S^0_n(v_i)=\emptyset$, for ${\tilde Z}_{ni}>t_{0n}$.
Denote ${\tilde Z}_{ni}^0:={\tilde Z}_{ni}{\mathbb I}_{\{{\tilde Z}_{ni}\le t_{0n}\}}=|S_n^0(v_i)|\sqrt{n/m}$.
Write
\begin{displaymath}
\varkappa^{1+\alpha}=2c_2/c_1,
\qquad
a_0:=c_0,
\qquad
a_{i+1}=a_i\varkappa,
\qquad
i=0,1,\dots,
\end{displaymath}
and note that $\varkappa^{1+\alpha}\ge 2$. Let ${\tilde Y}_n$ be a discrete random variable with values $0, a_0,a_1,\dots, a_{j_n}$, where
$j_n+1=\max\{i:\, a_i\le t_{0n}\}$,
and with probabilities $\PP({\tilde Y}_n=a_j)={\tilde p}_j$, defined by
\begin{displaymath}
{\tilde p}_j:=c_1a_j^{-1-\alpha}-c_2a_{j+1}^{-1-\alpha}=c_1a_j^{-1-\alpha}/2,
\qquad
j=0,1,\dots, j_n.
\end{displaymath}
Put $\PP({\tilde Y}_{n}=0)=1-{\tilde p}_{0}-{\tilde p}_{1}-\dots-{\tilde p}_{j_n}$.
Note that ${\tilde Y}_n$ is stochastically smaller than ${\tilde Z}_{ni}^0$, for every $1\le i\le n$. Indeed, (\ref{Y1})
implies, for $j=0,1,\dots, j_n$,
\begin{eqnarray}\nonumber
\PP(a_j<{\tilde Z}_{ni}^0\le a_{j+1})
&=&
\PP({\tilde Z}_{ni}>a_j)-\PP({\tilde Z}_{ni}>a_{j+1})
\\
\nonumber
&\ge&
 c_1a_j^{-1-\alpha}-c_2a_{j+1}^{-1-\alpha}
 \\
 \nonumber
 &=&
 {\tilde p}_j=\PP({\tilde Y}_n=a_j).
\end{eqnarray}

Let ${\tilde Y}_{n1},\dots, {\tilde Y}_{nn}$ be independent copies of ${\tilde Y}_n$ defined on the same probability space as
${\tilde Z}_{ni}^0, 1\le i\le n$ and such that almost surely
${\tilde Y}_{ni}\le {\tilde Z}_{ni}^0$, for every $1\le i\le n$ (such coupling is possible because  ${\tilde Y}_{ni}$ is stochastically
smaller than ${\tilde Z}_{ni}^0$). For $1\le i\le n$, let ${\tilde S}_n^0(v_i)$ be a random subset of $S_n^0(v_i)$ of size
$|{\tilde S}_n^0(v_i)|=\lfloor{\tilde Y}_{ni}\sqrt{m/n}\rfloor$ (which is uniformly distributed over the class of  subsets of $S_n^0(v_i)$
of  size $\lfloor{\tilde Y}_{ni}\sqrt{m/n}\rfloor$). Random subsets ${\tilde S}_n^0(v_i)$, $v_i\in V$ are independent
and identically distributed. They
define random intersection graph (denoted)
${\tilde G}_n^0$ which is a subgraph of $G_n^0$. It is easy to see that $\E {\tilde Y}_n^2\to \infty$. Therefore,
using Theorem 1 and Remark 2
of \cite{Bloznelis2008/1}, one can show that there exists $\rho\in (0,1)$ such that the number of vertices $C_1({\tilde G}_n^0)$
of the largest connected component of ${\tilde G}_n^0$ satisfies
\begin{equation}\label{C-G3}
\PP\bigl(|C_1({\tilde G}_n^0)|> \rho\,n\bigr)\to 1.
\end{equation}
The inclusions ${\tilde G}_n^0\subset G_n^0\subset G_n$ imply $|C_1({\tilde G}_n^0)|\le |C_1(G_n^0)|\le |C_1(G_n)|$ and,
by (\ref{C-G3}),
we obtain
\begin{equation}\label{C-G31}
\PP\bigl(|C_1(G_n)|> \rho\,n\bigr)
\ge
\PP\bigl(|C_1(G_n^0)|> \rho\,n\bigr)
\ge
\PP\bigl(|C_1({\tilde G}_n^0)|> \rho\,n\bigr)
\to 1.
\end{equation}
Note that (\ref{ro-11}) follows from (\ref{C-G31}).

Let us prove (\ref{L5-C1}).
 Denote $\delta=c_1/(12(1+c_0)^{1+\alpha})>0$. Write
 $t_*=(1+c_0)(2c_2/c_1)^{1/(1+\alpha)}$ and note that for large $n$ we have
$t_*<t_{0n}$. (\ref{Y1}) implies, for $1\le i\le n$,
\begin{equation}\label{DDELTTA}
\PP\bigl(1+c_0<{\tilde Z}_{ni}\le t_*\bigr)\ge \frac{c_1}{(1+c_0)^{1+\alpha}}-\frac{c_2}{t_*^{1+\alpha}}=6\delta.
\end{equation}
We assume  that $n$ is
large so that $\PP({\tilde Z}^0_{ni}>1)\ge 6 \delta$.
Denote
\begin{displaymath}
D=\cup_{v\in C_1(G_n^0)}S_n^0(v),
\qquad
d^*=\lfloor 2\delta \rho\sqrt{m\, n}\rfloor,
\qquad
k^*=\lfloor 2c_2(l_2(n))^{\alpha(1+\alpha)} \rfloor.
\end{displaymath}
Introduce the events
\begin{displaymath}
{\mathbb A}=\{\forall v\in V_n^0: \, S_n(v)\cap D\not=\emptyset\},
\qquad
{\mathbb B}=\{|D|> d^*\},
\qquad
{\mathbb D}=\{|V_n^0|\le k^*\}.
\end{displaymath}
Note that (\ref{L5-C1}) follows from the limit $\PP({\mathbb A})\to 1$, which itself follows from (\ref{L5-V0}) and the limits
\begin{equation}\label{DBA-1}
\PP({\mathbb B})\to 1,
 \qquad
\PP({\mathbb A}\cap {\mathbb B}\cap{\mathbb D})\to 1.
\end{equation}
Therefore, in order to prove (\ref{L5-C1}) it suffices to show (\ref{DBA-1}).

Let us  show the first  limit of (\ref{DBA-1}).
Denote
\begin{displaymath}
A=\sum_{v\in C_1(G_n^0)}|S_n(v)|,
\qquad
B=\sum_{\{v,u\}\subset C_1(G_n^0)}|S_n(v)\cap S_n(u)|.
\end{displaymath}
The obvious inequality $|D|\ge A-B$ combined with the bounds
\begin{equation}\label{AA-A}
\PP(B>\delta\rho \sqrt{mn})=o(1),
\qquad
\PP(A< 4\delta\rho \sqrt{mn})=o(1)
\end{equation}
implies $\PP({\mathbb B})\to 0$.
It remains to prove (\ref{AA-A}). It follows from (\ref{Y1}) that there exists a number $C>0$ (depending only on $\alpha, c_0,c_1,c_2$)
such that
$\E{\tilde Z}_{ni} \le C$ uniformly in $n>n_0$ and $1\le i\le n$. We have
\begin{displaymath}
\E B\le \sum_{1\le i<j\le n}\E|S_n(v_i)\cap S_n(v_j)|=\sum_{1\le i<j\le n}\frac{\E {\tilde Z}_{ni}{\tilde Z}_{nj}}{n}
\le \frac{C^2}{2}n.
\end{displaymath}
The bound $\E B=O(n)$ in combination with condition (i) of Theorem 1 implies the first bound of (\ref{AA-A}).
Let us prove the second bound of (\ref{AA-A}). Write $V_n^*=V_n\setminus V_n^0$.
We call a vertex $v\in V_n^*$  large if  ${\tilde Z}_n(v)>1$. Other vertices of  $V_n^*$ are called small.
 Let $N^*$
denote the number of large vertices in $V_n^*$.
Note that large vertices have  higher probabilities of belonging to $C_1(G_n^0)$ than small ones.
Therefore, the number ${\tilde N}$ of large vertices in $C_1(G_n^0)$ is stochastically larger than the number $N_0$
of large vertices in the simple random sample of size $|C_1(G_n^0)|$ drawn without replacement and with equal
probabilities from the set
$V_n^*$. The obvious inequality $A\ge {\tilde N}\sqrt{m/n}$ implies, for $s\ge 0$,
\begin{equation}\label{A-N-N-1}
\PP(A>s)\ge \PP({\tilde N}>s\sqrt{n/m})\ge \PP(N_0>s\sqrt{n/m}).
\end{equation}
We shall show that, for $s_n=4\delta\rho n$,
\begin{equation}\label{A-N-N-2}
\PP(N_0>s_n)\to 1.
\end{equation}
Introduce the events ${\mathbb H}=\{|C_1(G_n^0)|> \rho\,n\}$ and ${\mathbb B}^*=\{N^*\ge 5\delta n\}$ and denote
\begin{equation}\label{PN-1DG}
p(n)=\PP(\{N_0>s_n\}\cap {\mathbb D}\cap{\mathbb B}^*\cap{\mathbb H}).
\end{equation}
 By the total probability formula,
\begin{equation}\label{A-N-N-3}
p(n)=\sum_{h>\rho n}\sum_{b>5\delta n}\sum_{k\le k^*}p_{h,b,k}(n)\, \PP\bigl(|C_1(G_n^0)|=h,\, N^*=b,\, |V_n^*|=n-k\bigr).
\end{equation}
Here $p_{h,b,k}(n)$ denotes the conditional probability of the event $\{N_0>s_n\}$ given
$|C_1(G_n^0)|=h,\, N^*=b,\, |V_n^*|=n-k$. (\ref{Lema1-3}) applies to the hypergeometric probability
$p_{h,b,k}(n)=\PP(H_{h,b,n-k}>s_n)$ and, for large $n$,  shows $p_{h,b,k}(n)\ge 1-n^{-10}$. From (\ref{A-N-N-3}) we obtain
\begin{equation}\label{PN-2DG}
p(n)\ge \PP({\mathbb D}\cap{\mathbb B}^*\cap{\mathbb H})(1-n^{-10}).
\end{equation}
Note that the law of large numbers
combined with (\ref{L5-V0}) shows $\PP({\mathbb B}^*)\to 1$. This limit together with (\ref{C-G31}) and (\ref{L5-V0})
implies $\PP({\mathbb D}\cap{\mathbb B}^*\cap{\mathbb H})\to 1$. The latter limit, (\ref{PN-1DG}) and  (\ref{PN-2DG})
shows (\ref{A-N-N-2}). Finally, (\ref{A-N-N-2}) combined with (\ref{A-N-N-1}) implies the second bound of (\ref{AA-A}),
thus completing the proof of the limit $\PP({\mathbb B})\to 1$.

\smallskip

Let us  show the second  limit of (\ref{DBA-1}). The total probability formula gives
\begin{equation}\label{DBA-2}
\PP({\mathbb A}\cap {\mathbb B}\cap{\mathbb D})=\sum_{k\le k^*}\sum_{d>d^*}\PP_{k,d}({\mathbb A})\PP(|D|=d,\, |V_n^0|=k).
\end{equation}
Here $\PP_{k,d}$ denotes the conditional probability given $|D|=d,\, |V_n^0|=k$.
Let $S^*=\{|S_n(v)|, v\in V_n^0\}$ denote the collection of  sizes of sets $S_n(v)$ of vertices $v\in V_n^0$. Note that for
$|V_n^0|=k$, the collection $S^*=\{s_1,\dots, s_k\}$ is a multiset. We have
\begin{equation}\label{DBA-3}
\PP_{k,d}({\mathbb A})
=
\sum_{\{s_1,\dots, s_k\}}\PP_{k,d}\bigl({\mathbb A}\bigr|S^*=\{s_1,\dots, s_k\}\bigr)
\,
\PP_{k,d}(S^*=\{s_1,\dots, s_k\}).
\end{equation}
Here the sum is taken over all possible values $\{s_1,\dots, s_k\}$ of the multiset $S^*$ of cardinality $k$.
The identity
\begin{displaymath}
\PP_{k,d}\bigl({\mathbb A}\bigr|S^*=\{s_1,\dots, s_k\}\bigr)=\prod_{j=1}^k\PP(H_{s_j,d,m}\ge 1)
\end{displaymath}
combined with (\ref{Lema1-4}) implies, for large $n$, the inequality
$\PP_{k,d}\bigl({\mathbb A}\bigr|S^*=\{s_1,\dots, s_k\}\bigr)\ge 1-n^{-10}$ uniformly in $d,k$ and $s_1,\dots, s_k$
satisfying the inequalities $s_j\ge t_{0n}\sqrt{m/n}$, $d> d^*$, $k\le k^*$.
Now (\ref{DBA-3}) implies the inequality $\PP_{k,d}({\mathbb A})\ge 1-n^{-10}$, for $d>d^*$ and $k\le k^*$. Invoking the
latter inequality in
(\ref{DBA-2}) we obtain
\begin{displaymath}
\PP({\mathbb A}\cap{\mathbb B}\cap{\mathbb D})\ge (1-n^{-10})\PP( {\mathbb B}\cap{\mathbb D})=1-o(1).
\end{displaymath}
In the last step we used (\ref{L5-V0}) and the first bound of (\ref{DBA-1}). The proof of (\ref{DBA-1}) is complete.

\end{proof}

\begin{proof}[Proof of Theorem \ref{T1}] In the proof we use the approach developed in \cite{Reittu-Norros2002}, \cite{Reittu-Norros2004}, \cite{Reittu-Norros2005}.

Before the proof we introduce some notation.
Denote
\begin{eqnarray}\label{tk-1}
&&
t_0=n^{1/(1+\alpha)}l_2^{-\alpha}(n),
 \ \ \qquad
t_k=n^{\alpha^k/(1+\alpha)}l_2(n),
\quad
k=1,2, \dots,
\\
\nonumber
&&
k_*=\max\{k:\, n^{\alpha^k/(1+\alpha)}\ge 100+c_0\}.
\end{eqnarray}
We  use the following simple properties of the sequence $\{t_k\}$.
For $n>9$ we have
\begin{eqnarray}\label{tk-2}
&&
t_0 \, t_1/n=l_2^{1-\alpha}(n),
\qquad
t_k\,
t_{k-1}^{-\alpha}=l_2^{1-\alpha}(n),
\qquad
k=2,3,\dots,
\\
\label{k*}
&&
100\, l_2(n)< t_{k_*} <(100+c_0)^{1/\alpha}\, l_2(n),
\qquad
k_*\le l_2(n)/\ln(1/\alpha).
\end{eqnarray}

Given ${\cal U}\subset V_n$ we denote $S({\cal U})=\cup_{v\in {\cal U}}S_n(s)$.
Throughout the proof limits are taken as
$n\to\infty$. Given $n$, write $m=m_n$ and $T=T_n=n^{1/(1+\alpha)}\omega_n$. Fix $1<\tau<1+\alpha$. By $c_1^*, c_2^*,\dots$ we denote positive constants that may
depend only on $\alpha, \tau, c_0,c_1,c_2$.

 Let us prove (\ref{Egle}). Fix a maximal vertex $u_n$ of $G_n$.
We have
\begin{displaymath}
\PP\bigl(d(v_1,u_n)> k_*+\varepsilon l_2(n)\, \bigr| \,  d(v_1,u_n)<\infty \bigr)
=
\frac{
\PP\bigl(d(v_1,u_n)> k_*+\varepsilon l_2(n), \,  d(v_1,u_n)<\infty \bigr)
}
{\PP\bigl(d(v_1,u_n)<\infty \bigr)
}.
\end{displaymath}
In order to prove (\ref{Egle}) we shall show that
\begin{eqnarray}\label{Daumantas}
&&
\PP\bigl(d(v_1,u_n)> k_*+\varepsilon l_2(n), \,  d(v_1,u_n)<\infty \bigr)=o(1),
\\
\label{Kotryna}
&&
\liminf_n \PP\bigl(d(v_1,u_n)<\infty \bigr)>0.
\end{eqnarray}
Let us prove (\ref{Kotryna}). Write $C_1=C_1(G_n)$. It follows from (\ref{ro-11})
that $\PP(u_n\in C_1)=1-o(1)$.
Therefore, we have
\begin{equation}\label{T1-C1-1}
\PP(d(v_1,u_n)<\infty)\ge \PP(v_1,u_n\in C_1)=\PP(v_1\in C_1)-o(1).
\end{equation}
Inequalities (\ref{C-G31}) imply $\E|C_1|\ge\rho n (1-o(1))$ and, by symmetry, we obtain
 \begin{displaymath}
 \PP(v_1\in C_1)=n^{-1}\sum_{v\in V}\PP(v\in C_1)=n^{-1}\E|C_1|\ge \rho(1-o(1)).
 \end{displaymath}
This inequality combined with
(\ref{T1-C1-1}) implies (\ref{Kotryna}).

Let us prove (\ref{Daumantas}).
Introduce the   sets
\begin{eqnarray}\nonumber
&& {\cal U}_0=\{u_n\},
\qquad
{\cal U}_k=\{v_j:\ t_k\le {\tilde Z}_{nj}\le T\}, \qquad
k=1,2,\dots, k_*,
\\
\nonumber
&&
{\cal U}_*=\{v_j:\ 1\le {\tilde Z}_{nj}\le t_*\}.
\end{eqnarray}
 Denote $Q_k=\sum_{v\in {\cal U}_k}|S_n(v)|$ and $q_k=\E Q_k$.
 Introduce the events
\begin{eqnarray}\nonumber
&&
{\mathbb A}_0=\{t_0\le {\tilde Z}_{n}(u_n)\le T\},
\\
\nonumber
&&
{\mathbb A}_k=\{q_k/2\le Q_k\le (3/2)q_k \},
\qquad
k=1,2,\dots, k_*.
\\
\nonumber
&&
{\mathbb A}_{*1}=\bigl\{|{\cal U}_*|\ge 5n\delta\bigr\},
\qquad
{\mathbb A}_{*2}=\bigl\{|{\cal U}_{k_*}|\le n/l_2(n)\bigr\},
\end{eqnarray}
Here $\delta>0$ is defined in (\ref{DDELTTA}) above.
Denote ${\tilde {\mathbb A}}=\Bigl(\cap_{k=0}^{k_*}{\mathbb A}_k\Bigr)\cap{\mathbb A}_{*1}\cap{\mathbb A}_{*2}$. Let us show that
\begin{equation}\label{A-k*}
\PP({\tilde {\mathbb A}})\to 1.
\end{equation}
(\ref{A-k*}) follows from the limits
\begin{equation}\label{Limits}
\PP({\mathbb A}_{*i})\to 1,
\quad
i=1,2,
\qquad
\PP({\mathbb A}_0)\to 1,
\qquad
\PP\bigl(\cap_{1\le k\le k_*}{\mathbb A}_k\bigr)\to 1.
\end{equation}
 An application of Chebyshev's inequality to
the binomial random variables $|{\cal U}_{k_*}|$  and $|{\cal U}_*|$ gives the first limit of  (\ref{Limits}).
The second limit of (\ref{Limits}) is shown in (\ref{Lema3-1}). To show the third limit of (\ref{Limits})
we  write
\begin{displaymath}
1-\PP(\cap_{1\le k\le k_*}{\mathbb A}_k)
=
\PP(\cup_{1\le k\le k_*}{\overline {\mathbb A}}_k)
\le
\sum_{1\le k\le k_*}\PP({\overline {\mathbb A}}_k).
\end{displaymath}
Here ${\overline {\mathbb A}}_k$ denotes the event complement to ${\mathbb A}_k$. Combining the bound, which follows from (\ref{Lema3-2b}),
\begin{displaymath}
\PP({\overline {\mathbb A}}_k)\le c_1^*n^{(\alpha^k-1)(\tau-1)}\bigr(l_2(n)\bigr)^{(\alpha+1)(\tau-1)}
\end{displaymath}
and  the bound, see
(\ref{k*}), $k_*=O(l_2(n))$
we obtain $\sum_{1\le k\le k_*}\PP({\overline {\mathbb A}}_k)=o(1)$, thus showing the third limit of (\ref{Limits}). We arrive at
(\ref{A-k*}).

In the remaining part of the proof we shall assume that the event ${\tilde {\mathbb A}}$ holds.
Let $\PPP$, $\EEE$ and ${\tilde G}$ denote the conditional probability, the conditional expectation,
and the conditional random graph
$G_n$  given ${\tilde Z}_{n1},\dots, {\tilde Z}_{nn}$.
Write $V^*=V\setminus{\cal U}_{k_*}$ and  let $G^*$ denote the subgraph of ${\tilde G}$ induced by $V^*$.
Given $v\in V^*$ define $d_*(v)=\min\{d(w,v): w\in {\cal U}_{k_*}\}$.
We shall show that uniformly in ${\tilde Z}_{n1},\dots, {\tilde Z}_{nn}$ satisfying ${\tilde {\mathbb A}}$ and uniformly in
$v\in V^*$, $u\in {\cal U}_{k_*}$
\begin{eqnarray}\label{Dpadek1}
&&
\PPP(d_*(v)>\varepsilon l_2(n), d_*(v)<\infty)=o(1),
\\
\label{Dpadek2}
&&
\PPP(d(u,u_n)\ge k_*)=o(1).
\end{eqnarray}

It follows from (\ref{Dpadek1}) that a vertex $v\in V$ satisfying $d(v,u_n)<\infty$ finds whp a path of length at most $l_2(n)$
to a vertex $u \in{\cal U}_{k_*}$. (\ref{Dpadek2}) then applies to $u$ and together with (\ref{Dpadek1}) imply
\begin{equation}\label{DPadek3}
\PPP\bigl(d(v_1,u_n)> k_*+\varepsilon l_2(n), \,  d(v_1,u_n)<\infty \bigr)=o(1).
\end{equation}

The  bound (\ref{DPadek3}) combined with (\ref{A-k*}) shows (\ref{Daumantas}).
It remains to prove (\ref{Dpadek1}, \ref{Dpadek2}).

{\it Proof of (\ref{Dpadek1}).} For simplicity of notation we put $\varepsilon=1$. Given $v\in V^*$ denote
$L_*=\{v'\in V^*:\, d^*(v,v')\le l_2(n)\}$. Here $d^*$ denotes
the distance between vertices of the graph $G^*$.
Introduce the event ${\mathbb B}_*=\{|S(L_*)|\ge \delta\, l_2(n) \sqrt{m/n}\}$.
 The event
\begin{eqnarray}\nonumber
\{d_*(v)>l_2(n),\, d_*(v)<\infty\}&\subset& \{S(L_*)\cap S({\cal U}_{k_*})=\emptyset,\, |L_*|> l_2(n)\}
\\
\nonumber
&\subset&
\bigl(\{S(L_*)\cap S({\cal U}_{k_*})
=
\emptyset\}\cap {\mathbb B}_*\bigr)\cup \bigl({\overline {\mathbb B}}_*\cap \{|L_*|> l_2(n)\}\bigr).
\end{eqnarray}
Here ${\overline {\mathbb B}}_{*}$ denote the event complement to  ${\mathbb B}_{*}$. We have
\begin{equation}
\PPP\bigl(d_*(v)>l_2(n),\, d_*(v)<\infty\bigr)\le p'+p'',
\end{equation}
where
\begin{displaymath}
p'= \PPP\bigl(  \{ S(L_*)\cap S({\cal U}_{k_*})=\emptyset\}\cap {\mathbb B}_{*}),
\qquad
p''=\PPP\bigl( {\overline {\mathbb B}}_{*}\cap\{|L_*|>l_2(n)\}).
\end{displaymath}
We shall show that
\begin{equation}\label{p-k-1/2}
p'=o(1),
\qquad
 p''=o(1)
 \qquad
{ \text{as}}
\qquad
n\to \infty.
 \end{equation}
 Let us prove the first bound. (\ref{Lema3-2a}) and (\ref{k*}) imply $q_{k_*}>4c^*_2\sqrt{mn}/(l_2(n))^{\alpha}$.
 Invoking the inequality $Q_{k_*}\ge q_{k_*}/2$ and the inequality, which follows from Lemma 2,
  $\PPP(|S({\cal U}_{k_*})|\ge Q_{k_*}/2)=1-o(1)$
  we obtain the bound $1-\PPP({\mathbb B}')=o(1)$ for the event ${\mathbb B}'=\{|S({\cal U}_{k_*})|>c^*_2\sqrt{mn}/(l_2(n))^\alpha\}$.
 Therefore, we have
\begin{eqnarray}\nonumber
 p'&=& \PPP( \{ S(L_*)\cap S({\cal U}_{k_*})=\emptyset\}\cap {\mathbb B}'\cap {\mathbb B}_{*})+o(1)
\\
\nonumber
&\le&
\PPP( S(L_*)\cap S({\cal U}_{k_*})=\emptyset\bigr|{\mathbb B}',{\mathbb B}_{*})+o(1)
\\
\nonumber
&=&
o(1).
\end{eqnarray}
In the last step we applied (\ref{Lema1-4}) to the
random variable $H=\bigl|S(L_*)\cap S({\cal U}_{k_*})\bigr|$ conditionally, given $|S(L_*)|$ and $|S({\cal U}_{k_*})|$.

Let us show the second bound of (\ref{p-k-1/2}). Denote $k'=\lfloor l_2(n) \rfloor$.
 Let $\{v'_1,v'_2,\dots, v'_{n'}\}$ be an enumeration of elements of
$V^*$. We call $v'_i$ smaller than $v'_j$
whenever $i<j$. We call $v'\in V^*$ large if ${\tilde Z}_n(v')\ge 1$.
 Paint elements of $V^*$ white. Given $v\in V^*$ we
construct the
'breath first search' tree $T_v$ in $G^*$ with the root $v$ as follows.
Paint vertex $v$  black and write $\tau_0=v$. White vertices are checked in increasing order and those found adjacent to
$\tau_0$ are
painted black. Denote them $\tau_1<\tau_2<\cdots< \tau_{j_1}$. After all neighbours of $\tau_0$ have been found the vertex
$\tau_0$ is
called saturated. Then proceed recursively: take the first available black unsaturated vertex, say $\tau_i$ (here $i=\min\{j:\, \tau_j $ is black and unsaturated $\}$),
and find its neighbours among remaining white
vertices. Do this by checking white vertices in increasing order. After all white neighbours of $\tau_i$ have been found the vertex
$\tau_i$ is
called saturated, the neighbours are denoted $\tau_{j_{i-1}+1}<\tau_{j_{i-1}+2}<\cdots<\tau_{j_i}$ and painted black. We call
$\tau_i$ the parent
vertex of its children $\tau_{j_{i-1}+1},\dots, \tau_{j_i}$.
In this way we obtain the list $L=\{\tau_0,\tau_1,\dots\}$ of
vertices of the tree $T_v$. Denote $L_r=\{\tau_0,\dots, \tau_r\}$. Let ${\tilde N}$ denote the number of large vertices in the set $L_{k'}$.
We say that (player) $v$ receives a yellow card at step $r\ge 1$ if vertex $\tau_r$ is large and
$|S(L_{r-1})\cap S_n(\tau_{r})|\ge 2^{-1}|S_n(\tau_{r})|$.
The event that $v$ receives the first yellow card at step $r$
is denoted ${B}_r$. On the event ${\mathbb H}:=\bigl(\cap_{i=1}^{k'}{\overline {B}}_i\bigr)\cap\{{\tilde N}\ge 4\delta k'\}$ we have
\begin{equation}\label{see-above}
|S(L_{k'})|\ge  2^{-1}{\tilde N}\sqrt{m/n}\ge \delta \, l_2(n)\sqrt{m/n}.
\end{equation}
Note that the inequality $|L_*|> l_2(n)$ implies $|L|\ge k'+1$.
 Therefore, we have
 \begin{displaymath}
p''
 \le
 \PPP({\overline {\mathbb B}}_{*}\cap\{|L|>k'+1\}).
\end{displaymath}
Furthermore, for $|L|\ge k'+1$,
the inclusion
 $L_{k'}\subset L_*$ implies $|S(L_{k'})|\le |S(L^*)|$ and  in view of (\ref{see-above}) we conclude that events ${\mathbb H}$ and
${\overline {\mathbb B}}_{*}\cap\{|L|>k'+1\}$  do not intersect. We have
\begin{displaymath}
\PPP({\overline {\mathbb B}}_{*}\cap\{|L|>k'+1\})=\PPP({\overline {\mathbb B}}_{*}\cap\{|L|>k'+1\}\cap{\overline {\mathbb H}})
\le
p^*_1+p^*_2,
\end{displaymath}
where
\begin{displaymath}
p^*_1:=\PPP(\{|L|>k'+1\}\cap \{{\tilde N}< 4\delta k'\}),
\qquad
p^*_2:=\PPP(\cup_{r=1}^{k'}{\mathbb B}_r).
 \end{displaymath}
In order to prove the bound $p''=o(1)$ we shall show that $p^*_i=o(1)$, $i=1,2$.

Write
\begin{equation}\label{DKazuPd}
p^*_1=\PPP(|L|>k'+1){\tilde p},
\qquad
{\tilde p}:=\PPP({\tilde N}< 4\delta k'\,\bigl| \,|L|>k'+1).
\end{equation}
Since large vertices have higher probabilities to join the list $L$ than the other vertices we conclude that
the random variable ${\tilde N}$  is stochastically larger than the number $N_0$ of large vertices in the simple random
sample of size $k'+1$ drawn without replacement and with equal probabilities from the  set $V^*$.
In particular, we have
\begin{equation}\label{DKazmuPd}
{\tilde p}\le \PPP(N_0<4\delta k').
\end{equation}
 Note that
on the event ${\mathbb A}_{*1}\cap{\mathbb A}_{*2}$ we have $\E N_0\ge 5\delta (k'+1)$ and $\Var N_0=O(k')$.
Therefore, Chebyshev's inequality implies $\PPP(N_0<4\delta k')=O(1/k')=o(1)$.
This bound  combined with (\ref{DKazuPd}) and (\ref{DKazmuPd})
implies the bound $p^*_1=o(1)$.

In order to prove the bound $p^*_2=o(1)$ we
write
$p^*_2\le\sum_{r=1}^{k'}\PPP({\mathbb B}_r)$ and show that
\begin{equation}\label{xxn-10}
\PPP({\mathbb B}_r)\le n^{-10},
\end{equation}
for every $r$ and large $n$.
Before the proof of (\ref{xxn-10}) we introduce some notation.
For $i\ge 1$ denote $W_i=W\setminus S(L_{i-1})$, $S'(\tau_i)=S_n(\tau_i)\setminus S(L_{i-1})$, $m_i=|W_i|$, $s_i=|S_n(\tau_i)|$,
$s'_i=|S'(\tau_i)|$. Put $W_0=W$.
Fix $r\ge 1$.
Let $\tau_{r^*}$ denote the parent vertex of $\tau_r$. Denote $D_r=S(L_{r-1})\cap W_{r^*}$ and $d_r=|D_r|$.
We have $\PPP({\mathbb B}_r\bigr| W_{r^*}, D_r, S'(t_{r^*}), s_r)\le p_*$, where
\begin{equation}\label{P*}
p_*:=\PPP_{r^*}\bigl(|D_r\cap S_n(\tau_r)|\ge 2^{-1}s_r\,\Bigr|\, S_n(\tau_r)\cap S'(\tau_{r^*})\not=\emptyset\bigr).
\end{equation}
Here $\PPP_{r^*}$ denotes the conditional probability $\PPP$ given $W_{r^*}, D_r, S'(t_{r^*}), s_r$.
Note that in (\ref{P*}) values of all random variables are fixed (given), but $S_n(\tau_r)$ which is a random set
uniformly distributed in the class of subsets
of $W_{r^*}$ of given size $s_r$ satisfying  $s_r\ge \sqrt{m/n}$ (because $\tau_r$ is a large vertex).
It follows from (\ref{lydeka2}) that for large $n$ we have
\begin{equation}\label{obuoliai}
p_*\le e^{-s_r/8}(1+16m_{r^*}/s_r^2)\le e^{-8^{-1}\sqrt{m/n}}(1+16n)<n^{-10}.
\end{equation}
In the last step we applied condition (i) of Theorem 1.
(\ref{obuoliai}) implies (\ref{xxn-10}) thus completing the proof of (\ref{p-k-1/2}).
We arrive to (\ref{Dpadek1}).

{\it Proof of (\ref{Dpadek2}).}
Given $u_0'\in {\cal U}_{k_*}$ finds a neighbour in ${\cal U}_{k_*-1}$, say, $u_1'$ with
probability at least
\begin{displaymath}
\min_{u\in {\cal U}_{k_*}}\PPP\bigl(S_n(u)\cap S({\cal U}_{k_*-1})\not=\emptyset\bigr)=:p_0^*.
\end{displaymath}
 Similarly, given
$u'_j\in {\cal U}_{k_*-j}$ finds a neighbour in ${\cal U}_{k_*-j-1}$, say $u'_{j+1}$, with probability
at least
\begin{displaymath}
\min_{u\in {\cal U}_{k_*-j}}\PPP\bigl(S_n(u)\cap S({\cal U}_{k_*-j-1})\not=\emptyset\bigr)
=:p_j^*,
\end{displaymath}
 and so on.
In this way we may construct
a
 path (namely, $u_0', u_1',u_2',\dots, u'_{k_*}=u_n$) of length at most $k_*$ connecting
 $u_n$ with an
 arbitrary vertex $u_0'$
 from ${\cal U}_{k_*}$.
The probability that such a construction fails is at most
$\sum_{j=0}^{k_*-1}(1-p_j^*)$. In particular, for any given $u\in {\cal U}_{k_*}$, we have
\begin{displaymath}
\PPP(d(u,u_n)> k_*)\le \sum_{j=0}^{k_*-1}(1-p_j^*).
\end{displaymath}
In order to prove (\ref{Dpadek2}) we shall show that, for some $c_4^*>0$ and large $n$,
\begin{equation}\label{D-1}
1-p_j^*
\le
e^{-c_4^*(l_2(n))^{1-\alpha}}+n^{-2},
\qquad
0\le j\le k_*-1.
\end{equation}

Fix $1\le i\le k_*-1$  and $u\in {\cal U}_{i+1}$. On the event ${\tilde {\mathbb A}}$ we have, for large $n$,
\begin{equation}\label{dddd1}
Q_{i}
\ge
\frac{q_{i}}{2}
\ge
\frac{c_1}{4}\frac{1+\alpha}{\alpha}\frac{\sqrt{mn}}{t_i^{\alpha}},
\end{equation}
where the second inequality of (\ref{dddd1}) follows from (\ref{Lema3-2a}). Denote $c_4^*=\frac{c_1}{16}\frac{1+\alpha}{\alpha}$
and introduce the event  ${\mathbb B}=\{|S({\cal U}_{i})|\ge 2c_4^*\sqrt{m\,n}\, t_{i}^{-\alpha}\}$.
It follows from Lemma 2 (applied to $\gamma_1=1/10$ and $\gamma_2=1/2$) that
\begin{equation}\label{JMJ+}
1-n^{-2}\le \PPP(|S({\cal U}_{i})|\ge Q_{i}/2) \le \PPP({\mathbb B}).
\end{equation}
Here in the last step we invoke (\ref{dddd1}).
 Next we apply (\ref{Lema1-4}) to the hypergeometric random variable
$H=|S_n(u)\cap S({\cal U}_{i})|$, where $|S_n(u)|$ and $|S({\cal U}_{i})|$ are given  and satisfy
$|S_n(u)|\ge t_{i+1}\sqrt{m/n}$ and $|S({\cal U}_{i})|\ge 2c_4^*\sqrt{m\,n}\, t_{i}^{-\alpha}$. We obtain
\begin{eqnarray}\nonumber
\PPP\bigl(S_n(u)\cap S({\cal U}_{i})=\emptyset
\,
\bigr| {\mathbb B}\bigr)
&\le&
\exp\bigl\{-c_4^*\frac{t_{i+1}}{t_{i}^{\alpha}}\bigr\}
\\
\label{S-d}
&=&
\exp\{-c_4^*(l_2(n))^{1-\alpha}\}.
\end{eqnarray}
Combining (\ref{JMJ+}) and (\ref{S-d}) we obtain (\ref{D-1}), for $j=k_*-i-1$ satisfying $0\le j\le k_*-2$.
The proof of (\ref{D-1}) for $j=k_*-1$ is similar but simpler.
We arrive to (\ref{Dpadek2}) thus completing the proof of (\ref{Egle}).

\smallskip

Let us prove (\ref{Egle44}). Denote $a_n=(1+\varepsilon/2)(1/\alpha)\ln(\ln(2+n))$ and introduce the events
\begin{displaymath}
{\mathbb D}=\{v_1,v_2\in C_1\},
\qquad
{\mathbb G}=\{d(v_1,v_2)> 2a_n\},
\qquad
 {\mathbb G}_i=\{d(v_i,u_n)> a_n\},
 \qquad
 i=1,2.
 \end{displaymath}
Note that (\ref {Egle44}) is equivalent to the limit $\PP({\mathbb G}|{\mathbb D})=o(1)$. In order to prove (\ref{Egle44})
we shall show that
that there exists $\rho>0$ such that
\begin{eqnarray}\label{Aidas1}
&&
\liminf_n\PP({\mathbb D})>\rho^2,
\\
\label{Aidas2}
&&
\PP({\mathbb G}\cap {\mathbb D})=o(1).
\end{eqnarray}
Let us prove (\ref{Aidas1}). It follows from the identity $|C_1|=\sum_{v\in V}{\mathbb I}_{\{v\in C_1\}}$, by the symmetry, that
\begin{eqnarray}\nonumber
\E|C_1|^2
&&
=\E\sum_{v\in V}{\mathbb I}^2_{\{v\in C_1\}}+\E\sum_{\{u,v\}\in V}{\mathbb I}_{\{u,v\in C_1\}}
\\
\nonumber
&&
=\E|C_1|+n(n-1)\PP({\mathbb D}).
\end{eqnarray}
This identity combined with  $|C_1|\le n$ and the inequality, which follows from (\ref{ro-11}), $\E|C_1|^2\ge n^2\rho^2(1-o(1))$
shows (\ref{Aidas1}).

Let us prove (\ref{Aidas2}). In view of (\ref{Remark-1}) in suffices to show that
$p:=\PP({\mathbb G}\cap {\mathbb D}\cap\{u_n\in C_1\})=o(1)$. We have $p\le p_1+p_2$, where
$p_i=\PP(d(v_i,u_n)>a_n,\, d(v_i,u_n)<\infty)$, $i=1,2$.
Finally, (\ref{Daumantas}) implies $p_i=o(1)$ thus completing the proof of (\ref{Aidas2}).

\end{proof}

\end{document}